\documentclass[11pt]{article}
\usepackage[margin=1in]{geometry}  
\usepackage{graphicx}              
\usepackage{amsmath, amssymb, amsthm, epsfig}               
\usepackage{enumerate}
\usepackage{amsfonts}              
\usepackage{amsthm}                
\usepackage{amsthm}
\usepackage{enumerate}
\theoremstyle{plain}
\usepackage{color}
\usepackage{hyperref}
\hypersetup{colorlinks,citecolor=blue}
\usepackage{cite}
\newtheorem{corollary}{Corollary}[section]
\newtheorem{dfn}[corollary]{Definition}

\newtheorem{lemma}[corollary]{Lemma}

\newtheorem{thm}[corollary]{Theorem}

\newfont{\sBlackboard}{msbm10 scaled 1200}

\newcommand{\mylabel}[1]{\label{#1}
    \ifx\undefined\stillediting
    \else \fbox{$#1$}\fi }
\newcommand{\BE}{\begin{equation}}

\newcommand{\EEQ}{\end{equation}}
\newcommand{\rfb}[1]{\mbox{\rm
        (\ref{#1})}\ifx\undefined\stillediting\else:\fbox{$#1$}\fi}

\newfont{\Blackboard}{msbm10 scaled 1200}

\newfont{\roma}{cmr10 scaled 1200}
\def\R{{\mathbb R} }

\def\o{\omega}

\def\G{\Gamma}

\newcommand{\bb}{\begin{equation}}
\newcommand{\bbb}{\end{equation}}

\newcommand{\mm}    {{\hbox{\hskip 0.5pt}}}

\newcommand{\bluff} {{\hbox{\raise 15pt \hbox{\mm}}}}
\newcommand{\e}      {{\varepsilon}}
\renewcommand{\l }{\lambda }

scaled\magstep2

\def\G{W^{s,G}(\mathbb{R}^N)}

\newcommand{\ds}{\displaystyle}

\makeatletter
\def\section{\@startsection {section}{1}{\z@}{-3.5ex plus -1ex minus
        -.2ex}{2.3ex plus .2ex}{\large\bf}}
\makeatother
\numberwithin{equation}{section}

\begin{document}

\title{The concentration-compactness principle for fractional Orlicz-Sobolev spaces}

\author{Sabri Bahrouni,	Olimpio Miyagaki}

\author{Sabri Bahrouni$^{\mathrm{a}}$\thanks{The author is supported by FAPESP Proc 2023/04515-7, Email: sabri.bahrouni@fsm.rnu.tn},	
		Olimpio Miyagaki$^{\mathrm{b}}$\thanks{
				The author is supported by CNPq Proc 303256/2022-2 and FAPESP Proc 2022/16407-1, Email: olimpio@ufscar.br.}  \\
			{\small $^{\mathrm{a}}$  Department of Mathematics, Faculty of Sciences, University of Monastir, Tunisia}\\	
			{\small $^{\mathrm{b}}$Departamento de Matem\'{a}tica, Universidade Federal de S\~ao Carlos S\~ao Carlos,}\\
            {\small  SP, CEP:13565-905,  Brazil}
}
\date{}

\maketitle

\begin{abstract}
In this paper, we delve into the well-known concentration-compactness principle in fractional Orlicz-Sobolev spaces, and we apply it to establish the existence of a weak solution for a critical elliptic problem involving the fractional $g-$Laplacian.
\end{abstract}



\section{Introduction}

\subsection{An Overview}

The paper by Bonder and Salort \cite{Bonder-Salort1} establishes a connection between fractional-order theories and Orlicz-Sobolev settings. In their work, the authors introduce the concept of fractional order Orlicz-Sobolev space associated with a Young function $G$ and a fractional parameter $0 < s < 1$, denoted as $W^{s,G}(\mathbb{R}^N)$. This space is defined as follows:
$$
W^{s,G}(\R^N)=\bigg{\{}u\in L^{G}(\R^N):\ \int_{\R^N}\int_{\R^N}G\bigg{(}\frac{|u(x)-u(y)|}{|x-y|^{s}}\bigg{)}\frac{dxdy}{|x-y|^{N}}<\infty\bigg{\}}.
$$
Here, $L^G(\mathbb{R}^N)$ is the set of measurable functions $u:\mathbb{R}^N\to \mathbb{R}$ such that $\int_{\mathbb{R}^N} G(|u|) < \infty$.

To describe phenomena in this setting, the appropriate operator is the fractional $g$-Laplacian, which is introduced in \cite{Bonder-Salort1} and defined as:

\begin{equation}\label{op0}
(-\Delta_g)^s u = \text{p.v.} \int_{\mathbb{R}^N} g\left(\frac{|u(x) - u(y)|}{|x - y|^s}\right) \frac{u(x) - u(y)}{|u(x) - u(y)|} \frac{dy}{|x - y|^{N + s}},
\end{equation}

Here, $G$ is a Young function such that $g = G'$, and the operator is defined using the principal value for any fractional parameter $s\in (0,1).$ In the context of the fractional $g$-Laplacian, it encompasses several significant specific instances, notably, the "fractional Laplace operator" (achieved when selecting $G(t)=t^2$) and the "fractional $p$-Laplace operator" (attained when choosing $G(t)=t^p$ for $p>1$). The focus on studying operators with nonlocal and nonstandard growth (the fractional $g$-Laplacian), has grown significantly. Numerous researchers have devoted their attention to establishing fundamental properties of the space $W^{s,G}$, provided  the following structural condition satisfied:
\begin{equation} \label{G1} \tag{$G_1$}
1< p^-\leq \frac{tg(t)}{G(t)} \leq p^+<\infty ,\quad \text{for all }t>0.
\end{equation}
 While it's impractical to cover all the papers in this field, we provide some references, and interested readers can explore further by referencing the papers mentioned \cite{ACPS1,ACPS2,ACPS3,ACPS4,ABS1,ABS2,Sabri1,Sabri2,Sabri3,Sabri4,Bonder-Salort1,NBS1,Salort1}.\\

The issue of lack of compactness has garnered significant attention in the field of fractional and nonlocal elliptic studies. As a result, there has been substantial growth in the relevant literature, with a plethora of significant papers published recently on this subject. If you're looking for an introduction to this class of problems and an extensive list of references, we recommend consulting the recent book \cite{Pucci-G}.\\

Sobolev-type compact embedding theorems are an important tool in proving existence results for wide classes of variational problems. For instance, the classical Rellich-Kondrachov theorem says that the following embedding is compact:
 $$
W^{1,p}(\Omega)\to \begin{cases}
                                 L^{\frac{Np}{N-p}}(\Omega) & \quad 1\leq p<N \\
                                  L^q(\Omega)\ \forall q<\infty & \quad p=N \\
                                  L^\infty(\Omega) & \quad p>N.
\end{cases}
$$
Here the boundedness of $\Omega$ is a fundamental assumption. For instance, unlike bounded domains, no compact embedding is available for $W^{1,p}(\R^N)$. In this context, in \cite{Sabri2}, we established
a compact embedding theorems for the fractional Orlicz-Sobolev space were Orlicz spaces are in target. More precisely, for any Young function $B$ such that $G_{*}$ is essentially stronger than $B$, denoted $B\prec\prec G_{*}$, we proved the compact embedding of $W^{s,G}(\Omega)$ into  $L^{B}(\Omega)$, where the critical function $G_*$ is given by
$$
 G_{*}^{-1}(t)=\displaystyle\int_{0}^{t}\frac{G^{-1}(\tau)}{\tau^{\frac{N+s}{N}}}d\tau.
$$

An enhanced embedding was addressed in \cite{ACPS1}, where it furnishes us with the ideal Orlicz target space within the Sobolev embedding framework for space $W^{s,G}(\R^N)$. This optimal space is constructed based on the Young function $G_{\frac{N}{s}}$, which is associated with the parameters $G, N,$ and $s$ in the following manner. Let $s \in(0,1)$ and let $G$ be a Young function such that
\begin{equation}\label{G2}\tag{$G_2$}
  \int^{\infty}\left(\frac{t}{G(t)}\right)^{\frac{s}{N-s}} d t=\infty,
\end{equation}
and
\begin{equation}\label{G3}\tag{$G_3$}
  \int_0\left(\frac{t}{G(t)}\right)^{\frac{s}{N-s}} d t<\infty.
\end{equation}
Then, $G_{\frac{N}{s}}$ is given by
\begin{equation}\label{Critical-function}
  G_{\frac{N}{s}}(t)=G\left(H^{-1}(t)\right) \text { for } t \geq 0,
\end{equation}
where the function $H:[0, \infty) \rightarrow[0, \infty)$ obeys
$$
H(t)=\left(\int_0^t\left(\frac{\tau}{G(\tau)}\right)^{\frac{s}{N-s}} d \tau\right)^{\frac{N-s}{s}} \text { for } t \geq 0 .
$$

Under the conditions \eqref{G2} and \eqref{G3}, we have the embedding
$$
W^{s,G}(\R^N)\hookrightarrow L^{G_{\frac{N}{s}}}(\R^N),
$$
and there exists a constant $C=C(N, s)$ such that
\begin{equation}\label{Inequality-embedding}
  \|u\|_{G_{\frac{N}{s}}} \leq C\|u\|_{s,G},
\end{equation}
for every function $u \in W^{s,G}(\R^N)$. Moreover, $L^{G_{\frac{N}{s}}}\left(\mathbb{R}^N\right)$ is the oprimal target space in inequality  among all Orlicz spaces. In particular, \eqref{Inequality-embedding} yields
\begin{equation}\label{Sobolev-constant}
  S:=\inf_{u\in W^{s,G}(\R^N)\backslash\{0\}}\frac{\|u\|_{s,G}}{\|u\|_{G_{\frac{N}{s}}}}>0.
\end{equation}

In the recent paper \cite{Sabri3}, the first author studied the existence of solutions of a Quasilinear elliptic problem in the whole space  $\mathbb{R}^N,$ $N$ an integer $N\geq 2$, implying the lack of compactness. More precisely, we considered the following nonlinear fractional $g-$Laplacian elliptic problem
\begin{align} \label{m.equation}
\begin{cases}
(-\Delta_g)^s u +g(u)\frac{u}{|u|}=f(u)&\text{ in } \mathbb{R}^N\\
u\in W^{s,G}(\mathbb{R}^N).
\end{cases}
\end{align}

As we mentioned before, the unboundedness of the domain generally prevents the study of these types of problems by general methods of nonlinear analysis due to the lack of compactness.
It has been observed that by restricting the study to sub-spaces formed by functions having some symmetries of the problem, some forms of compactness were obtained.
Denote by
$$
W^{s,G}_{rad}(\mathbb{R}^{N})=\{u\in W^{s,G}(\mathbb{R}^{N}):\ u\ \text{is radially symmetric}\}.
$$

However, when $F\sim G_{\frac{N}{s}}$ in \eqref{m.equation} where $F^{\prime}=f$ (critical type problems), the existence problem becomes much more delicate. In the case $s = 1$ and $G(t) = t^2$ we recover the famous Yamab\'{e} equation appearing in Riemannian geometry and studied by Aubin \cite{Aubin} and then by Br\'{e}zis-Nirenberg \cite{BN}.
What is remarkable is that, generally, these are almost the only two ways to lose compactness ({\it the unboundedness of the domain $\&$ the critical type problems}).

One highly significant tool for addressing critical issues was formulated by P.L. Lions in his renowned work cited as \cite{Lions1}. P.L. Lions introduced the concentration-compactness principle (referred to as CCP for brevity), which involves the examination of compactness limitations in bounded sequences within the function space $W^{1,p}(\Omega)$.

To keep the introduction concise and straightforward, we will not provide a comprehensive explanation of this principle. However, interested readers can refer to the following papers for more detailed information \cite{Bonder1, Bonder2, Bonder3, Ho-Kim, Fu1}.

\subsection{Main results}

We define the fractional $(s, G)$-gradient of a function $v\in W^{s,G}(\R^N)$ as
\begin{equation}\label{gradient}
 \mathcal{D}^s v(x)=G( |v|)+\int_{\R^N}^{} G( |D^s v|) \frac{dy}{|x-y|^N}.
\end{equation}

The main result of this paper reads:

\begin{thm}\label{CCP}
Let $s\in(0,1)$. Assume that $G$ is a Young function satisfying conditions \eqref{G1}-\eqref{G3}, and let $G_{\frac{N}{s}}$ be the Young function defined as in \eqref{Critical-function}.

Let $(u_n)_{n\in \mathbb{N}}$ be a bounded sequence in $W^{s,G} (\R^N)$ such that
$$
\begin{aligned}
   & u_n \rightharpoonup u\quad \text{in}\quad W^{s,G} (\R^N), \\
   & \mathcal{D}^s u_n \overset{\ast}{\rightharpoonup}\mu\quad \text{ in }\quad \mathcal{R}(\R^N),\\
   &G_{\frac{N}{s}} (|u_n|) \overset{\ast}{\rightharpoonup} \nu \text{ in }\mathcal{R}(\R^N),
   \end{aligned}
$$
where $\mathcal{R}(\R^N)$ is the space of all signed finite Radon measures on $\R^N.$

Then, there exist positive numbers $(\nu_i)_{i\in I}$ and $(\mu_i)_{i\in I}$, where $I$ is countable index set, such that
\begin{equation}\label{CCP1}\tag{CCP1}
  \nu = G_{\frac{N}{s}} (|u|) + \sum_{i\in I}^{} \nu_i \delta_{x_i},
\end{equation}
\begin{equation}\label{CCP2}\tag{CCP2}
  \mu \geq \mathcal{D}^s u+ \sum_{i\in I}^{} \mu_i \delta_{x_i},
\end{equation}
\begin{equation}\label{CCP3}\tag{CCP3}
  S \min \left\{ \nu_i^{\frac{1}{p_\ast^-}}, \ \nu_i^{\frac{1}{p_\ast^+}} \right\} \leq \max \left\{ \mu_i, \ \mu_i^{\frac{p^-}{p^+}}, \ \mu_i^{\frac{p^+}{p^-}} \right\},
\end{equation}
where $S$ defined in \eqref{Sobolev-constant} and $p_*^{\pm}=\displaystyle\frac{Np^\pm}{N-sp^\pm}$.
Moreover, if we define
$$
\begin{aligned}
   & \nu _\infty= \lim_{R \rightarrow+\infty} \limsup_{n \rightarrow \infty} \int_{|x|>R} G_{\frac{N}{s}} (|u_n|)  dx, \\
   & \mu_\infty = \lim_{R \rightarrow+\infty} \limsup_{n \rightarrow \infty} \int_{|x|>R} \mathcal{D}^s u_n dx.\\
\end{aligned} $$
Then,
\begin{equation}\label{CCP4}\tag{CCP4}
  \limsup_{n \rightarrow \infty} \int_{\R^N}^{} G_{\frac{N}{s}} (|u_n|)  dx=\nu (\R^N)+\nu_\infty,
\end{equation}
\begin{equation}\label{CCP5}\tag{CCP5}
\limsup_{n \rightarrow \infty} \int_{\R^N}^{}  \mathcal{D}^s u_n dx=\mu(\R^N) +\mu_\infty,
\end{equation}
\begin{equation}\label{CCP6}\tag{CCP6}
 S \min \left\{ (\nu_\infty)^{\frac{1}{p_\ast^-}}, \ (\nu_\infty)^{\frac{1}{p_\ast^+}}  \right\} \leq \max \left\{ \mu_\infty, (\mu_\infty)^{\frac{p^-}{p^+}}, \ (\mu_\infty)^{\frac{p^+}{p^-}} \right\}.\end{equation}
\end{thm}

As an application, we shall consider the following nonlinear fractional $g-$Laplacian elliptic problem
\begin{align} \label{m.equation}\tag{$\Sigma$}
\begin{cases}
(-\Delta_g)^s u +g(u)\frac{u}{|u|}=g_*(u)\frac{u}{|u|}+\lambda f(u)&\text{ in } \mathbb{R}^N\\
u\in W^{s,G}(\mathbb{R}^N),
\end{cases}
\end{align}
where $\lambda>0$ is a real parameter and $g_*$ is such that
\begin{equation}\label{g*}
  G_{\frac{N}{s}}(t)=\int_{0}^{t}g_*(\tau)d\tau.
\end{equation}
The nonlinearity $f\colon\mathbb{R}\rightarrow\mathbb{R}$ assumed to be continuous
\begin{equation}\label{f1}
  \lim_{t\to 0}\frac{f(t)}{g(t)}=0,
\end{equation}
\begin{equation}\label{f2}
\limsup_{t\to +\infty}\frac{|f(t)|}{m(|t|)}<+\infty,
\end{equation}
where $m\colon (0,+\infty)\to \mathbb{R}$ is continuous function satisfying:
\begin{equation}\label{m1}
0<m^-\leq\frac{tm(t)}{M(t)}\leq m^+,\quad\text{for all } t>0,
\end{equation}
 where $p^+<m^-<m^+<p^-_*:=\frac{Np^-}{N-sp^-}$ and $M(t)=\int^t_0m(s)\ ds$ is a Young function.

There is $\theta>p^+,$ such that
\begin{equation}\label{f3}
0<\theta F(t)\leq f(t)t,\quad\text{for all }t\in\mathbb{R}\backslash\{0\},
\end{equation}
where
\begin{equation}\label{primitive}
    F(t)=\int_{0}^{t}f(\tau)\,d\tau.
\end{equation}

Our existence result for Problem \eqref{m.equation} is the following

\begin{thm}\label{Existence}
Assume that $G$ is a Young function satisfying conditions \eqref{G1}-\eqref{G3}. Suppose that \eqref{f1}-\eqref{f3} hold. Then, there exists $\lambda_*>0,$ such that for $\l>\l_*$ problem \eqref{m.equation} has at least a nontrivial weak solution.
\end{thm}

\section{Preliminaries}

In this section, we will define the fractional order Orlicz-Sobolev spaces and we introduce some technical results that will be used throughout the paper.
\subsection{Young functions}

An application $G\colon\R_{+}\to \R_{+}$ is said to be a  \emph{Young function} if it admits the integral formulation $G(t)=\int_0^t g(\tau)\,d\tau$, where the right continuous function $g$ defined on $[0,\infty)$ has the following properties:
\begin{align}
&g(0)=0, \quad g(t)>0 \text{ for } t>0 \label{g0},\\
&g \text{ is non-decreasing on } (0,\infty) \label{g2},\\
&\lim_{t\to\infty}g(t)=\infty  \label{g3}.
\end{align}
From these properties, it is easy to see that a Young function $G$ is continuous, nonnegative, strictly increasing and convex on $[0,\infty)$.

The following  properties on Young functions are well-known. See for instance \cite{Bonder-Salort1} for the proofs.

\begin{lemma} \label{lema.prop}
Let $G$ be a Young function satisfying \eqref{G1} and $a,b\geq 0$. Then
\begin{align}
  &\min\{ a^{p^-}, a^{p^+}\} G(b) \leq G(ab)\leq   \max\{a^{p^-},a^{p^+}\} G(b),\label{L1}\\
  &\forall \delta>0,\ \exists C_\delta\ \text{such that}\quad  G(a+b)\leq C_\delta G(a)+(1+\delta)^{p^+}G(b),\label{L2}\\
	&G \text{ is Lipschitz continuous}. \label{L_3}
 \end{align}
\end{lemma}




The \emph{complementary Young function} $\tilde G$ of a Young function $G$ is defined as
$$
\tilde G(t):=\sup\{tw -G(w): w>0\}.
$$
The functions $G$ and $\tilde{G}$
are complementary each other and satisfy the inequality below
\begin{equation} \label{ineb}
\tilde{G}\left(g(t)\right)\leq G(2t),\qquad \text{for all } t > 0.
\end{equation}
We also have a Young type inequality given by
\begin{equation} \label{Young}
ab\leq G(a)+\tilde G(b)\qquad \text{for all }a,b\geq 0.
\end{equation}
 Moreover, it is not hard to see that $\tilde G$ can be written as
\begin{equation} \label{xxxx}
\tilde G(t)=\int_0^t \tilde{g}(\tau)\,d\tau,
\end{equation}
where $\tilde{g}(t)=\sup\{s:\ g(s)\leq t\}$. If $g$ is continuous then $\tilde{g}$ is the inverse of $g$.

In order to state some embedding results for fractional Orlicz-Sobolev spaces we recall that given two Young functions $A$ and $B$, we say that \emph{$B$ is essentially stronger than $A$} or equivalently that \emph{$A$ decreases essentially more rapidly than $B$}, and denoted by $A\prec \prec B$, if for each $a>0$ there exists $x_a\geq 0$ such that $A(x)\leq B(ax)$ for $x\geq x_a$. This is the case if and only if for
every positive constant $\lambda$,
\begin{equation} \label{do2}
 \lim\limits_{\substack{  t\to \infty}} \frac{A(\lambda t)}{B( t)}=0.
 \end{equation}

Let's introduce the notion of the Matuszewska-Orlicz function and the Matuszewska-Orlicz index.
\begin{dfn}
   Given a Young function $G$, we define the associated Matuszewska-Orlicz function as
$$
M(t, G):=\limsup _{\tau \rightarrow \infty} \frac{G(\tau t)}{A(\tau)} .
$$

When no confusion arises, we will simply denote $M(t)=M(t, G)$.

The Matuszewska-Orlicz index is then defined as
$$
p_{\infty}(G):=\lim _{t \rightarrow \infty} \frac{\ln M(t, G)}{\ln t}=\inf _{t>0} \frac{\ln M(t, G)}{\ln t} .
$$
Again, when no confusion arises, we will simply denote $p_{\infty}=p_{\infty}(G)$.
\end{dfn}

The main feature that we use in this article is the fact that, if $G$ verifies the $\Delta_2$-condition, then the index $p_{\infty}$ is finite and for any $\varepsilon>0$, there exists $t_0>0$ such that
$$
t^{p_{\infty}} \leq M(t, G) \leq t^{p_{\infty}+\varepsilon}, \text { for } t \geq t_0.
$$

See \cite{Arriagada} for this fact and more properties of this index.

 It is also easy to check that if $G$ satisfy \eqref{G1}), then $p^{-} \leq p_{\infty} \leq p^{+}$, and that
$$
\min \left\{t^{p^{+}}, t^{p^{-}}\right\} M(\tau) \leq M(\tau t) \leq \max \left\{t^{p^{+}}, t^{p^{-}}\right\} M(\tau) .
$$

If $G$ verifies \eqref{G1}, then $M$ is a Young function (\cite[Lemma 2.9]{Bonder3}).

\subsection{Fractional Orlicz-Sobolev spaces}

Given a Young function $G$ such that $G'=g$ and a parameter $s\in(0,1)$. We consider the spaces
\begin{align*}
&L^G(\R^N) :=\left\{ u\colon \R^N \to \R \text{ measurable  such that }  \Phi_{G}\left(\lambda u\right) < \infty \text{ for some }\lambda>0 \right\},\\
&W^{s,G}(\R^N):=\left\{ u\in L^G(\R^N) \text{ such that } \Phi_{s,G}\left(\lambda u\right) < \infty \text{ for some }\lambda>0 \right\},
\end{align*}
where the modulars $\Phi_{G}$ and $\Phi_{s,G}$ are defined as
\begin{align*}
&\Phi_{G}(u):=\int_{\R^N} G(|u(x)|)\,dx,\\
&\Phi_{s,G}(u):=
  \iint_{\R^N\times\R^N} G( |D^s u(x,y)|)  \,d\sigma,
\end{align*}
where  the \emph{$s-$H\"older quotient} is defined as
$$
D^s u(x,y):=\frac{u(x)-u(y)}{|x-y|^s},
$$
being $d\sigma(x,y):=\frac{ dx\,dy}{|x-y|^N}$.
These spaces are endowed with the so-called \emph{Luxemburg norms}
\begin{align*}
&\|u\|_{G} := \inf\left\{\lambda>0\colon \Phi_{G}\left(\frac{u}{\lambda}\right)\le 1\right\},\\
&\|u\|_{s,G} := \inf\left\{\lambda>0\colon \ \rho\left(\frac{u}{\lambda}\right)\leq 1\right\},
\end{align*}
where
$$
\rho(u)=\Phi_{G}(u)+\Phi_{s,G}(u).
$$
The Orlicz space can be defined with respect to a complete measure $\mu$ in $\mathbb{R}^N$, and the associated norm is represented as $\|\cdot\|_{G, d\mu}$.

We have the following H\"older's type inequality given by
$$
\int_{\R^N} |uv|\,dx \leq 2\|u\|_{G} \|v\|_{\tilde{G}},
$$
for all $u\in L^G(\R^N)$ and $v\in L^{\tilde G}(\R^N)$.

Under the assumption \eqref{G1}, the space $W^{s,G}(\R^N)$ is a reflexive Banach space. Moreover $C_c^\infty(\R^N)$ is dense in $W^{s,G}(\R^N)$. Also, the spaces $L^G(\R^N)$ and $W^{s,G}(\R^N)$ coincide with the set of measurable functions $u$ on $\R^N$ such that $\Phi_{G}\left( u\right) < \infty$ and the set of functions in $L^G(\R^N)$  such that $ \Phi_{s,G}\left( u\right) < \infty$ respectively. See \cite[Proposition 2.11]{Bonder-Salort1} and \cite[Theorem 4.7.3]{FS} for details.

The space of fractional Orlicz-Sobolev functions is the appropriated one to define the \emph{fractional $g-$Laplacian operator}
$$
(-\Delta_g)^s u :=2 \,\text{p.v.} \int_{\R^N} g( |D^s u|) \frac{D^s u}{|D^s u|} \frac{dy}{|x-y|^{N+s}},
$$
where \text{p.v.} stands for {\em in principal value}. This operator  is well defined between $W^{s,G}(\R^N)$ and its dual space $W^{-s,\tilde G}(\R^N)$ (see \cite{Bonder-Salort1} for details). In fact,
$$
\langle (-\Delta_g)^s u,v \rangle =   \iint_{\R^N\times\R^N} g(|D^s u|) \frac{D^s u}{|D^s u|}  D^s v \,d\mu,\quad\text{for all }v\in W^{s,G}(\R^N).
$$

A relation between modulars and norms holds. See \cite[Lemma 2.1]{Fukagai} and \cite[Lemma 13]{Sabri1}.
\begin{lemma}\label{modulars-norms}
   Let $G$ be a Young function satisfying \eqref{G1} and let $G_0(t)=\min\{t^{p^-},t^{p^+}\}$, $G_\infty(t)=\max\{t^{p^-},t^{p^+}\}$, for all $t\geq0$.
    Then
    \begin{itemize}
      \item[(i)] $G_0(\|u\|_{L^G(\R^N)})\leq\Phi_{G}(u)\leq G_\infty(\|u\|_{L^G(\R^N)})\ \text{for}\ u\in L^{G}(\R^N)$,
      \item[(ii)] $G_0(\|u\|_{s,G})\leq \rho(u) \leq G_\infty(\|u\|_{s,G})\ \text{for}\ u\in W^{s,G}(\R^N)$.
    \end{itemize}
 \end{lemma}

\section{Proof of the Concentration-Compactness Principles for $W^{s,G}(\mathbb{R}^N)$: Theorem \ref{CCP} }

\begin{proof}
  Let $\o_n =u_n-u$. Then $u_n\rightharpoonup 0$ in $W^{s,G}(\mathbb{R}^N)$. Invoking \eqref{Inequality-embedding}, we deduce that $w_n \rightarrow 0$ in $L_{loc}^B(\R^N)$ where $B  \prec\prec  G_{\frac{N}{s}}$. Hence, up to a subsequence, we have $w_n(x) \rightarrow 0$ a.e. $x \in \R^N.$ Using the fact that $ G_{\frac{N}{s}} ( |u_n|)\overset{\ast}{\rightharpoonup} \nu$ and \cite[Lemma 3.4]{Bonder3} or \cite[Theorem 3.3]{Sabri5}, we obtain
  $$
  \lim_{n \rightarrow +\infty} \left( \int_{\R^N}^{}\varphi G_{\frac{N}{s}} ( |u_n|)dx -\int_{\R^N}^{}\varphi G_{\frac{N}{s}} ( |w_n|)dx \right)= \int_{\R^N}^{}\varphi G_{\frac{N}{s}} ( |u|)dx
  $$
  for any $\varphi \in C^{\infty} _c (\R^N)$, from where the representation
  $$
   G_{\frac{N}{s}} ( |w_n|)  \overset{\ast}{\rightharpoonup} \nu -G_{\frac{N}{s}} ( |u|):= \bar{\nu} \text{ in } \mathcal{R}(\R^N).
  $$
  The sequence $\left\{ G( |w_n|)+\int_{\R^N}^{} G( |D^s w_n|) \frac{dy}{|x-y|^N} \right\}$ is bounded in $L^1(\R^N).$ So, up to a subsequence, we have
  $$
  G( |w_n|)+\int_{\R^N}^{} G( |D^s w_n|) \frac{dy}{|x-y|^N} \overset{\ast}{\rightharpoonup} \bar{\mu} \text{ in } \mathcal{R}(\R^N).
  $$
  Let $\phi \in C^{\infty} _c (\R^N)$  and let $R>2$  be such that $ supp(\phi) \subset B_R$ and $ d:= dist\big( B^c_R, supp(\phi) \big) \geq 1+\frac{R}{2}.$\\
 By \eqref{Inequality-embedding}, we have
 $$
 S\| \phi w_n \|_{G_{\frac{N}{s}} } \leq \| \phi w_n \|_{s,G}.
 $$
  Set $\bar{\nu}_n:= G_{\frac{N}{s}}( |w_n|), \ \bar{\mu}_n: G( |w_n|)+\ds\int_{\R^N}^{} G( |D^s w_n|) \frac{dy}{|x-y|^N}$,  $\lambda_n:= \| \phi w_n \|_{s,G}$ (we have temporarily modified the notation of $\mathcal{D}^s u_n$ to facilitate its usage.).
  From \eqref{L2} and Lemma \ref{modulars-norms}, we have
  \begin{equation}\label{PCCP2}
  \begin{aligned}
  1&=\rho \big( \frac{|\phi w_n|}{\l_n} \big)\\
  &\leq \int_{\R^N}^{}G\left( \frac{|\phi w_n|}{\l_n} \right) dx  +(1+\delta )^{p^+}\int_{\R^N}^{}\int_{\R^N}^{}G \bigg( \frac{|\phi(x)|}{\l_n} |D^s w_n| \bigg) \frac{dxdy}{|x-y|^N}\\
  & \ \ + C(\delta) \int_{\R^N}^{}\int_{\R^N}^{} G\bigg( \frac{|w_n(y)|}{\l_n} |D^s \phi| \bigg) \frac{dxdy}{|x-y|^N}\\
  & \leq\|\phi\|_\infty \bigg[\int_{\R^N}^{} G\left( \frac{|w_n|}{\l_n}\right) dx + (1+\delta )^{p^+}\int_{\R^N}^{}\int_{\R^N}^{}G \bigg( \frac{|D^s w_n|}{\l_n}  \bigg) \frac{dxdy}{|x-y|^N}\bigg] +c(\delta)I_n,
  \end{aligned}
   \end{equation}
where
$$
I_n:=  \int_{\R^N}^{}\int_{\R^N}^{} G\bigg( \frac{|w_n(y)|}{\l_n} |D^s \phi| \bigg) \frac{dxdy}{|x-y|^N}.
$$
If $\l_n >1,$ then $\frac{1}{\l_n}<1.$ Invoking Lemma \ref{modulars-norms}, we get
\begin{equation}\label{PCCP7}
\begin{aligned}
1 & \leq \frac{\|\phi\|_\infty}{\l_n} (1+\delta )^{p^+} \rho(w_n) +C(\delta)I_n\\
&\leq \frac{\|\phi\|_\infty}{\l_n} (1+\delta)^{p^+} \big(  1+\|w_n\|_{s,G}\big) +C(\delta)I_n.
\end{aligned}
\end{equation}
\underline{{\bf Claim 1:}}
$$
I_n \leq \frac{c}{\min (\l_n^{p^-}, \ \l_n^{p^+})} \left[\frac{1}{(\frac{R}{2})^{N+sp^-}}+\max \left\{ \frac{1}{R^{sp^-}}, \ \frac{1}{R^{sp^+}} \right\} + \max \left\{ \frac{1}{(2R)^{(s-1)p^-}},\frac{1}{(2R)^{(s-1)p^+} }\right\}\right].
$$
Let's prove the Claim.  From \eqref{L1}, we have
\begin{align}\label{PCCP1}
  I_n \leq & \int_{\R^N}^{} G_\infty\bigg( \frac{|w_n|}{\l_n}\bigg) \int_{\R^N} G(|D^s \phi|) \frac{dy}{ |x-y|^N} dx\nonumber\\
\leq & \frac{1}{\min \big(\l_n^{p^-}, \ \l_n^{p^+}\big)} \bigg[  \int_{B^c_R}^{} G_\infty\bigg( \frac{|w_n|}{\l_n}\bigg) \int_{B_R} G(|D^s \phi|) \frac{dy}{ |x-y|^N} dx\nonumber\\
& + \int_{B_R}^{} G_\infty\bigg( \frac{|w_n|}{\l_n}\bigg) \int_{B^c_R} G(|D^s \phi|) \frac{dy}{ |x-y|^N} dx\nonumber\\
& +\int_{B_R}^{} G_\infty\bigg( \frac{|w_n|}{\l_n}\bigg) \int_{B_R} G(|D^s \phi|) \frac{dy}{ |x-y|^N} dx\bigg].
\end{align}
Now, we will estimate each integral in the right hand side of \eqref{PCCP1}.\\
Since $G_\infty\prec\prec G_{\frac{N}{s}}$, then by \eqref{Inequality-embedding} and the boundness of $(w_n)_n$ in $\G$, we can find $c>0$ such that
$$
\int_{\R^N}^{} G_\infty\big(|w_n| \big) dx \leq c.
$$
Here, $c$ is a positive constant independent of both $n$ and $R$, whereas $c(R)$ represents a positive constant that remains independent of $n$ but can vary from line to line.\\
Using \eqref{L1}, we have
$$
\begin{aligned}
&  \int_{B^c_R}^{} G_\infty\big( |w_n|\big) \int_{B_R} G(|D^s \phi|) \frac{dy}{ |x-y|^N} dx\\
& = \int_{B^c_R}^{} G_\infty\big( |w_n|\big) \int_{supp (\phi)} G\bigg(\frac{| \phi(y)|}{|x-y|^s}\bigg) \frac{dy}{ |x-y|^N} dx\\
&\leq \|\phi\|_\infty\int_{B^c_R}^{} G_\infty\big( |w_n|\big) \int_{supp (\phi)}G(1) G_\infty\bigg(\frac{1}{|x-y|^s}\bigg) \frac{dy}{ |x-y|^N} dx\\
&\leq \|\phi\|_\infty G(1)\int_{B^c_R}^{} G_\infty\big( |w_n|\big) \int_{supp (\phi)} G_\infty\bigg((\frac{R}{2})^{-s}\bigg)(\frac{R}{2})^{-N} dy dx\\
&\leq \|\phi\|_\infty G(1)\int_{B^c_R}^{} G_\infty\big( |w_n|\big) \int_{supp (\phi)} (\frac{R}{2})^{-N-sp^-} dy dx\\
& \leq \frac{\|\phi\|_\infty G(1)}{(\frac{R}{2})^{N+sp^-}} \big| B_R \big| \int_{\R^N}^{} G_\infty\big( |w_n|\big)dx  \leq \frac{c}{(\frac{R}{2})^{N+sp^-}}.
\end{aligned}
$$
For the second integral in the right-hand side of \eqref{PCCP1}, we note that
$$
\begin{aligned}
\int_{B_R^c}^{} G \bigg( \big| D^s \phi \big| \bigg)\frac{dy}{ |x-y|^N}=& \int_{B_R^c}^{} G \bigg( \frac{\big| \phi (x) \big|}{|x-y|^s} \bigg)\frac{dy}{ |x-y|^N}\\
&\leq \|\phi\|_\infty G(1)\int_{B^c_R}^{} G_\infty\bigg(\frac{1}{|x-y|^s}\bigg) \frac{dy}{ |x-y|^N}\\
&\leq \|\phi\|_\infty G(1)\int_{|z|>1}^{} G_\infty\bigg(\frac{1}{|z|^s}\bigg) \frac{dz}{ |z|^N}\\
& \leq c \max \big\{ \frac{1}{R^{sp^-}}, \ \frac{1}{R^{sp^+}} \big\}.
\end{aligned}
$$
Then
$$
\begin{aligned}
\int_{B_R}^{} G_\infty \big(|w_n|\big)\int_{B_R^c} G(|D^s \phi|) \frac{dy}{ |x-y|^N} dx&\leq c \max \big\{ \frac{1}{R^{sp^-}}, \ \frac{1}{R^{sp^+}} \big\} \int_{B_R}^{} G \big(|w_n|\big)dx \\
&\leq  c \max \big\{ \frac{1}{R^{sp^-}}, \ \frac{1}{R^{sp^+}} \big\}.
\end{aligned}
$$
Let's estimate the last integral in the right-hand side of \eqref{PCCP1}, for $x\in B_R,$ we have
$$
\begin{aligned}
\int_{B_R} G(|D^s \phi|) \frac{dy}{ |x-y|^N} dx &\leq \| \nabla \phi \|_\infty  \int_{B_R} G \bigg( \frac{1}{|x-y|^{s-1}}\bigg) \frac{dy}{ |x-y|^N}\\
& =  \| \nabla \phi \|_\infty  \int_{|z|\leq 2R} G \bigg( \frac{1}{|z|^{s-1}}\bigg) \frac{dz}{ |z|^N}\\
&= \| \nabla \phi \|_\infty  \max \left\{  \frac{1}{\big( 2R\big)^{(s-1)p^-}}, \  \frac{1}{\big( 2R\big)^{(s-1)p^+}} \right\}.
\end{aligned}
$$
This yields
$$
\int_{B_R}^{} G_\infty\bigg( |w_n|\bigg) \int_{B_R} G(|D^s \phi|) \frac{dy}{ |x-y|^N} dx \leq c \max \left\{  \frac{1}{\big( 2R\big)^{p^-(s-1)}}, \  \frac{1}{\big( 2R\big)^{p^+(s-1)}} \right\}.
$$
Combining the last three estimates, we obtain the claim.

Letting $R\rightarrow +\infty$ in \eqref{PCCP7}, from Claim 1 and the boundness of $(w_n)$ in $W^{s,G}(\R^N)$, we deduce that
$$ 1\leq \frac{\|\phi \|_\infty}{\l_n} (1+\e)^{p^+} \big( 1+\| w_n\|_{s,G}\big)\quad \Rightarrow\quad \l_n \leq c.$$
Remember that  we considered $\l_n>1$, for the case  $\l_n <1$, it is enough to take $c=1.$

Thus $\l_n$ is a bounded sequence in $\R$. Consequently, we can assume, up to a subsequence, that there exists $\lambda_* \geq 0$ such that
$$
\lim_{n\to \infty} \l_n=\l_*.
$$
Suppose that $\l_*>0.$ Again, from \eqref{PCCP2}, we have
$$
\begin{aligned}
  1 &\leq (1+\delta)^{p^+}\bigg[ \int_{\R^N}G_\infty\left( \frac{|\phi|}{\l_n} \right)G(|w_n|) dx  +\int_{\R^N}^{}\int_{\R^N}^{}G_\infty \bigg( \frac{|\phi(x)|}{\l_n}\bigg) G( |D^s w_n|) \frac{dxdy}{|x-y|^N}\bigg] + C(\delta) I_n\\
   &\leq (1+\delta)^{p^+}\int_{\R^N}G_\infty\left( \frac{|\phi|}{\l_n} \right)d\bar{\mu}_n + C(\delta) I_n.
\end{aligned}
$$
Letting $n\to +\infty$ and $R\to +\infty$ in the last inequality, we obtain
$$
\begin{aligned}
  1 &\leq (1+\delta)^{p^+}\int_{\R^N}G_\infty\left( \frac{|\phi|}{\l_*} \right)d\bar{\mu}.
\end{aligned}
$$
Now, letting $\delta\to0^+,$ then
$$
1\leq \int_{\R^N}G_\infty\left( \frac{|\phi|}{\l_*} \right)d\bar{\mu},
$$
and so
$$
\lambda_*\leq \|\phi\|_{G_\infty, d\bar{\mu}}.
$$
We argue as the proof of \cite[Lemma 4.4]{Bonder3}, we get
\begin{equation}\label{PCCP5}
  \|\phi w_n\|_{G_{\frac{N}{s}}}\geq \|\phi \|_{M_{\frac{N}{s}},d\bar{\mu}},
\end{equation}
where $M_{\frac{N}{s}}$ is the Matuszewska-Orlicz function associated to $G_{\frac{N}{s}}$.
Then
\begin{equation}\label{PCCP3}
  S\|\phi \|_{M_{\frac{N}{s}},d\bar{\mu}}\leq \|\phi \|_{G_\infty,d\bar{\mu}}.
\end{equation}
Now, we can apply \cite[Lemma 4.7]{Bonder3}, we get
$$
\bar{\nu}=\sum_{i\in I}^{} \nu_i \delta_{x_i},
$$
which means $\nu=G_{\frac{N}{s}}(u)+\sum_{i\in I}^{} \nu_i \delta_{x_i}.$ Thus the proof of \eqref{CCP1}.

Next, we obtain \eqref{CCP3}. Let $\phi\in C^\infty_c (\R^N)$ be such that $0\leq \phi \leq 1,$ $\phi\equiv1$ on $B_1,$ $supp(\phi)\subset B_2.$
For each $i\in I$ and for $\e>0,$ define $\phi_{\e,i}(x):=\phi\left(\frac{x-x_i}{\e}\right).$
Thus $\phi_{\e,i}\in C^\infty_c (\R^N),$ $0\leq \phi_{\e,i}\leq1$, $\phi_{\e,i}\equiv1$ on $B_\e(x_i),$ $supp(\phi_{\e,i})\subset B_{2\e}(x_i).$

Combining \eqref{Inequality-embedding} and \eqref{PCCP5}, we get
\begin{equation}\label{PCCP4}
  S\|\phi_{\e,i} \|_{M_{\frac{N}{s}},d\bar{\nu}}\leq \|\phi_{\e,i} u_n\|_{s,G}.
\end{equation}
Applying Lemma \ref{modulars-norms} with the Young function $M_{\frac{N}{s}}$, we have
$$
\begin{aligned}
\|\phi_{\e,i} \|_{M_{\frac{N}{s}},d\bar{\nu}}&\geq \min\bigg\{\left(\int_{B_{2\e}(x_i)}M_{\frac{N}{s}}(|\phi_{\e,i}(x)|)d\bar{\nu}\right)^{\frac{1}{p^-_*}},
\left(\int_{B_{2\e}(x_i)}M_{\frac{N}{s}}(|\phi_{\e,i}(x)|)d\bar{\nu}\right)^{\frac{1}{p^+_*}} \bigg\}\\
&\geq\min\bigg\{\left(\bar{\nu}(B_\e(x_i))\right)^{\frac{1}{p^-_*}},
\left(\bar{\nu}(B_\e(x_i))\right)^{\frac{1}{p^+_*}} \bigg\}.
\end{aligned}
$$
Thus, we obtain a lower bound of the left-hand side of \eqref{PCCP4} as follows:
$$
\limsup_{\e\to 0^+}S\|\phi_{\e,i} \|_{M_{\frac{N}{s}},d\bar{\nu}}\geq S\min\left\{\left(\nu_i\right)^{\frac{1}{p^-_*}},
\left(\nu_i\right)^{\frac{1}{p^+_*}} \right\}.
$$
Let's now look for an upper bound of the right-hand side of \eqref{PCCP4}. Set $\l_{n,\e}:=\|\phi_{\e,i} u_n \|_{s,G}.$ We can see that
 $$
 0<\liminf_{n\to \infty}\l_{n,\e}=:\lambda_{*,\e}\leq \lambda_0\quad\text{for some}\quad \lambda_0>0\ \text{ and for any}\ \e \ \text{small enough}.
 $$
For any $\e>0,$ we have
\begin{equation}\label{PCCP13}
\begin{aligned}
1&=\rho\left(\frac{|\phi_{\e,i} u_n|}{\l_{n,\e}}\right)\\
&\leq  \int_{\R^N}^{}G\left( \frac{|\phi_{\e,i}u_n|}{\l_{n,\e}} \right) dx  +2\int_{B_{2\e}(x_i)}^{}\int_{B_{2\e}^c(x_i)}^{}G ( |D^s(\phi_{\e,i}u_n)| ) \frac{dydx}{|x-y|^N}\\
  & + \int_{B_{2\e}(x_i)}^{}\int_{B_{2\e}(x_i)}^{}G ( |D^s(\phi_{\e,i}u_n)| ) \frac{dydx}{|x-y|^N}\\
  &\leq   \int_{\R^N}^{}G_\infty\left( \frac{|\phi_{\e,i}|}{\l_{n,\e}} \right)G(|u_n|) dx  +2\int_{B_{2\e}(x_i)}^{}\int_{B_{2\e}^c(x_i)}^{}G_\infty\left( \frac{|u_n(x)|}{\l_{n,\e}} \right)G ( |D^s(\phi_{\e,i})| ) \frac{dydx}{|x-y|^N}\\
  & +(1+\delta)^{p^+} \int_{B_{2\e}(x_i)}^{}\int_{B_{2\e}(x_i)}^{}G_\infty\left( \frac{|\phi_{\e,i}|}{\l_{n,\e}} \right)G ( |D^s u_n| ) \frac{dydx}{|x-y|^N}\\
  &+ C(\delta) \int_{B_{2\e}(x_i)}^{}\int_{B_{2\e}(x_i)}^{}G_\infty\left( \frac{|u_n(y)|}{\l_{n,\e}} \right)G ( |D^s \phi_{\e,i}| ) \frac{dydx}{|x-y|^N}\\
  &\leq \frac{(2+C(\delta)}{\min\left\{\l_{n,\e}^{p^-},\l_{n,\e}^{p^+}\right\}}\int_{\R^N}^{}\int_{\R^N}^{}G_\infty\left(|u_n(x)| \right)G ( |D^s \phi_{\e,i}| ) \frac{dydx}{|x-y|^N}\\
  &+\frac{(1+\delta)^{p^+}}{\min\left\{\l_{n,\e}^{p^-},\l_{n,\e}^{p^+}\right\}}\int_{\R^N} G_\infty\left( |\phi_{\e,i}(x)| \right)\mathcal{D}^s u_n(x) dx.
\end{aligned}
\end{equation}
Letting $n\to\infty$ in the previous inequality, we obtain
\begin{equation}\label{PCCP6}
\begin{aligned}
1&\leq  \frac{(2+C(\delta)}{\min\left\{\l_{*,\e}^{p^-},\l_{*,\e}^{p^+}\right\}}\limsup_{n\to\infty}\int_{\R^N}^{}\int_{\R^N}^{}G_\infty\left(|u_n(x)| \right)G ( |D^s \phi_{\e,i}| ) \frac{dydx}{|x-y|^N}\\
  &+\frac{(1+\delta)^{p^+}}{\min\left\{\l_{*,\e}^{p^-},\l_{*,\e}^{p^+}\right\}}\int_{\R^N} G_\infty\left( |\phi_{\e,i}(x)| \right)d\mu.
\end{aligned}
\end{equation}

\underline{{\bf Claim 2:}}
$$
\lim_{\e\to 0^+}\limsup_{n\to\infty}\int_{\R^N}^{}\int_{\R^N}^{}G_\infty\left(|u_n(x)| \right)G ( |D^s \phi_{\e,i}| ) \frac{dydx}{|x-y|^N}=0.
$$

Claim 2 can be inferred by conducting a meticulous examination of the proof presented in \cite[Lemma 4.4]{Ho-Kim}.

From Claim 2 and letting $\e\to 0^+$ in \eqref{PCCP6}, we get
\begin{equation}\label{PCCP8}
\min\left\{\l_{*,0}^{p^-},\l_{*,0}^{p^+}\right\}\leq (1+\delta)^{p^+}\max\left\{\mu_i^{p^-},\mu_i^{p^+}\right\},
\end{equation}
where $\l_{*,0}=\lim_{\e\to0^+}\lambda_{*,\e}$ and $\mu_i=\lim_{\e\to0^+}\mu(B_{2\e}(x_i))$.

Letting $\delta\to 0^+$ in \eqref{PCCP8}, we obtain
$$
\l_{*,0}\leq \max\left\{\left(\max\left\{\mu_i^{p^-},\mu_i^{p^+}\right\}\right)^{\frac{1}{p^-}},
\left(\max\left\{\mu_i^{p^-},\mu_i^{p^+}\right\}\right)^{\frac{1}{p^+}}\right\}\leq \max\left\{\mu_i,\mu_i^{\frac{p^-}{p^+}},\mu_i^{\frac{p^+}{p^-}}\right\}.
$$
Thus the upper bound of \eqref{PCCP4}, and the proof of \eqref{CCP3} is complete.

{\color{red}Now, let's prove} \eqref{CCP2}.\\

It remains to prove \eqref{CCP4}-\eqref{CCP6}. Let $\phi\in C^\infty_c (\R^N)$ be such that $0\leq \phi\leq 1,$ $\phi\equiv0$ on $B_1,$ $\phi\equiv1$ on $\R^N\backslash B_2.$ For each $R>0,$ define $\phi_R(x)=\phi\left(\frac{x}{R}\right).$ Thus $\phi_R\in C^\infty_c (\R^N),$ $0\leq \phi_R\leq 1,$ $\phi_R\equiv0$ on $B_R$ and $\phi_R\equiv1$ on $B_{2R}.$

We then write that
$$
\int_{\R^N}\mathcal{D}^s u_n dx=\int_{\R^N}\phi_R(x)\mathcal{D}^s u_n dx+\int_{\R^N}(1-\phi_R(x))\mathcal{D}^s u_n dx.
$$
Observe that for all $n\in\mathbb{N}$ and $R>0$, we have
$$
\int_{B_{2R}^c}\mathcal{D}^s u_n dx\leq \int_{\R^N}\phi_R(x)\mathcal{D}^s u_n dx\leq \int_{B_{R}^c}\mathcal{D}^s u_n dx.
$$
So that
\begin{equation}\label{PCCP9}
  \mu_\infty = \lim_{R \rightarrow+\infty} \limsup_{n \rightarrow \infty} \int_{\R^N}\phi_R(x)\mathcal{D}^s u_n dx.
\end{equation}
On the other hand, since $1-\phi_R\in C^\infty_c (\R^N)$, we have
$$
\lim_{n\to\infty}\int_{\R^N}(1-\phi_R(x))\mathcal{D}^s u_n dx=\int_{\R^N}(1-\phi_R(x))d\mu.
$$
Since $\Phi_R\to 0$ pointwise, it follows from the Dominated Convergence Theorem that
$$
\lim_{R\to\infty} \int_{\R^N}\phi_R(x)d\mu=0.
$$
Hence
\begin{equation}\label{PCCP10}
  \lim_{R \rightarrow+\infty} \limsup_{n \rightarrow \infty}\int_{\R^N}(1-\phi_R(x))\mathcal{D}^s u_n dx=\mu(\R^N).
\end{equation}
Combining \eqref{PCCP9} and \eqref{PCCP10} we deduce \eqref{CCP5}. The proof of \eqref{CCP4} is similar.

Finally, we still have to prove \eqref{CCP6}. From \eqref{Inequality-embedding}, we have
\begin{equation}\label{PCCP11}
   S\|\phi_{R}u_n \|_{G_{\frac{N}{s}}}\leq \|\phi_{R} u_n\|_{s,G}.
\end{equation}
Using Lemma \ref{modulars-norms}, we have
$$
\begin{aligned}
\|\phi_{R}u_n \|_{G_{\frac{N}{s}}}&\geq \min\bigg\{\left(\int_{B_{R}^c}G_{\frac{N}{s}}(|\phi_{R}(x)u_n(x)|)dx\right)^{\frac{1}{p^-_*}},
\left(\int_{B_{R}^c}G_{\frac{N}{s}}(|\phi_{R}(x)u_n(x)|)dx\right)^{\frac{1}{p^+_*}} \bigg\}\\
&\min\bigg\{\left(\int_{B_{2R}^c}G_{\frac{N}{s}}(|u_n(x)|)dx\right)^{\frac{1}{p^-_*}},
\left(\int_{B_{2R}^c}G_{\frac{N}{s}}(|u_n(x)|)dx\right)^{\frac{1}{p^+_*}} \bigg\}.
\end{aligned}
$$
Thus,
\begin{equation}\label{PCCP12}
  \liminf_{R\to +\infty}\limsup_{n\to\infty}\|\phi_{R}u_n \|_{G_{\frac{N}{s}}}\geq\min\left\{(\nu_\infty)^{\frac{1}{p_*^-}},(\nu_\infty)^{\frac{1}{p_*^-}}\right\}
\end{equation}
Next, we will evaluate the right-hand side of \eqref{PCCP11}. Set $\sigma_{n,R}:=\|\phi_{R} u_n\|_{s,G}$. We can see that
$$
0< \limsup_{n\to\infty}\sigma_{n,R}=:\sigma_{*,R}\leq \sigma \quad\text{for some }\ \sigma>0\ \text{and}\ R\ \text{large enough}.
$$
Argue as in \eqref{PCCP13}, we obtain
\begin{equation}\label{PCCP14}
  \begin{aligned}
  1&=\rho\left(\frac{|\phi_{R} u_n|}{\sigma_{n,R}}\right)\\
  &\leq \frac{(2+C(\delta)}{\min\left\{\sigma_{n,R}^{p^-},\sigma_{n,R}^{p^+}\right\}}\int_{\R^N}^{}\int_{\R^N}^{}G_\infty\left(|u_n(x)| \right)G ( |D^s \phi_{R}| ) \frac{dydx}{|x-y|^N}\\
  &+\frac{(1+\delta)^{p^+}}{\min\left\{\sigma_{n,R}^{p^-},\sigma_{n,R}^{p^+}\right\}}\int_{\R^N} G_\infty\left( |\phi_{R}(x)| \right)\mathcal{D}^s u_n(x) dx.
  \end{aligned}
\end{equation}

Letting $n\to\infty$ in the previous inequality, we obtain
\begin{equation}\label{PCCP15}
\begin{aligned}
1&\leq  \frac{(2+C(\delta)}{\min\left\{\sigma_{*,R}^{p^-},\sigma_{*,R}^{p^+}\right\}}\limsup_{n\to\infty}\int_{\R^N}^{}\int_{\R^N}^{}G_\infty\left(|u_n(x)| \right)G ( |D^s \phi_{R}| ) \frac{dydx}{|x-y|^N}\\
  &+\frac{(1+\delta)^{p^+}}{\min\left\{\sigma_{*,R}^{p^-},\sigma_{*,R}^{p^+}\right\}}\int_{\R^N} G_\infty\left( |\phi_{R}(x)| \right)d\mu.
\end{aligned}
\end{equation}

\underline{{\bf Claim 3:}}
$$
\lim_{R\to+\infty}\limsup_{n\to\infty}\int_{\R^N}^{}\int_{\R^N}^{}G_\infty\left(|u_n(x)| \right)G ( |D^s \phi_{R}| ) \frac{dydx}{|x-y|^N}=0.
$$
Claim 3 can be inferred by conducting a meticulous examination of the proof presented in \cite[4.5]{Ho-Kim}.

From Claim 3 and letting $R\to +\infty$ in \eqref{PCCP15}, we get
\begin{equation}\label{PCCP16}
\min\left\{\sigma_{*,\infty}^{p^-},\sigma_{*,\infty}^{p^+}\right\}\leq (1+\delta)^{p^+}\max\left\{\mu_\infty^{p^-},\mu_\infty^{p^+}\right\},
\end{equation}
where $\sigma_{*,\infty}=\lim_{R\to +\infty}\sigma_{*,R}$.

Letting $\delta\to 0^+$ in \eqref{PCCP16}, we obtain
$$
\sigma_{*,\infty}\leq \max\left\{\mu_\infty,\mu_\infty^{\frac{p^-}{p^+}},\mu_\infty^{\frac{p^+}{p^-}}\right\}.
$$
Thus the upper bound of \eqref{PCCP11}, and the proof of \eqref{CCP6} is complete.

\end{proof}

\section{Existence of weak solutions for \eqref{m.equation}}

In this section, for the sake of simplicity, we will assume that the structural assumptions necessary for Theorem \ref{Existence}  are in effect, without the need for further elaboration.
\begin{dfn}\label{weak.solution}
  We say that $u\in W^{s,G}(\R^N)$ is a \emph{weak solution} of \eqref{m.equation} if
$$
\langle (-\Delta_g)^s u,v \rangle +\int_{\R^N} g(|u|)\frac{u}{|u|}v\, dx= \int_{\R^N} g_*(|u|)\frac{u}{|u|}v\, dx+\lambda\int_{\R^N} f(u)v\,dx
$$
for all $v\in W^{s,G}(\R^N)$.
\end{dfn}

 Certainly, the weak solutions of equation \eqref{m.equation} correspond precisely to the critical points of the Euler-Lagrange functional denoted as $I_\lambda$, which is associated with equation \eqref{m.equation}. In this context, the functional $I_\lambda$, defined on $ W^{s,G}(\mathbb{R}^{N})$ by

\begin{equation}\label{functional}
  I_\lambda(u)=\Phi_{s,G}(u)+\Phi_G(u)-\Phi_{G_{\frac{N}{s}}}(u)-\mathcal{F}(u),
\end{equation}
where
\begin{equation}\label{func.1}
 \mathcal{F}(u):=\int_{\R^N}F(u)\,dx,\quad F\ \text{ is given in}\ \eqref{primitive},
\end{equation}
which is well defined and of class $C^1$.\\

It will be shown that the functional $I_\lambda$ possesses the requisite geometric attributes to establish the existence of a Palais–Smale sequence at specific energy levels, leveraging the mountain pass theorem initially formulated by Ambrosetti and Rabinowitz.

\begin{lemma}\label{geo.condition}
$~$
The functional $I_\lambda$ satisfies the mountain pass geometry, that is,
\begin{itemize}
  \item [(i)] There exist $\rho>0$ and $\delta_{\rho}>0$ such that $I_\lambda(u)\geq \delta_{\rho}$ for any $u\in W^{s,G}(\R^N)$ with $\|u\|_{s,G}=\rho.$
  \item [(ii)] There exists a strictly positive function $e\in W^{s,G}(\R^N)$ such that $\|e\|_{s,G}>\rho$, $\|e\|_{G_{\frac{N}{s}}}>0$ and $I_\l(e)< 0.$
\end{itemize}
\end{lemma}

\begin{proof}
 \begin{itemize}
  \item [(i)] From \eqref{f1}-\eqref{f3}, given $\epsilon>0$, there exists $C_\epsilon>0$ such that
  \begin{equation*}
  0\leq F(t)\leq \frac{\epsilon p^+}{\theta}G(\vert t\vert )+C_\epsilon M(\vert t\vert ),\quad\text{for all } t\in \mathbb{R},
  \end{equation*}
 we get
  \begin{align*}
     I_\l(u)&\geq \iint_{\R^N\times\R^N} G(|D_s u|)\frac{dxdy}{|x-y|^N} +\left(1-\frac{\epsilon p^+}{\theta}\right)\int_{\R^N} G(|u|)dx\\
     &-C_\epsilon\int_{\R^N} M(|u|)dx-\int_{\R^N} G_{\frac{N}{s}}(|u|)dx.
  \end{align*}
 Hence for $\epsilon$ small enough and according to Lemma \ref{modulars-norms}, there exist $C_1,\ C_2,\ C_3>0$ such that
 \begin{equation}\label{e1}
 I_\l(u)\geq C_1\left( G_0(\|u\|_{s,G})\right)-C_2 M_\infty(\|u\|_{M})-C_3\max\left\{\|u\|_{G_\frac{N}{s}}^{p^-_*},\|u\|_{G_\frac{N}{s}}^{p^-_*}\right\}.
 \end{equation}
 Choosing $\rho>0$ such that
  \begin{align*}
   \|u\|_{M}\leq C\|u\|_{G_\frac{N}{s}}\leq C^{'}\|u\|_{s,G}<\rho<1.
  \end{align*}
   By \eqref{e1}, we obtain
 \begin{equation*}
 I_\l(u)\geq C_1\|u\|_{s,G}^{p^+}-C_2\|u\|_{M}^{m^-}-C_3\|u\|_{G_\frac{N}{s}}^{p_*^-}
 \end{equation*}
 witch yields
 \begin{equation*}
 I_\l(u)\geq C_1^{'}\|u\|_{s,G}^{p^+}-C_2^{'}\|u\|_{s,G}^{m^-}-C_3^{'}\|u\|_{s,G}^{p_*^-}
 \end{equation*}
 for some positive constants $C_1^{'}$, $C_2^{'}$ and $C_3^{'}$. Since $0<p^+<m^-<p^-_*$, there exists  $\delta_{\rho}>0$ such that
 $$I_\l(u)\geq \delta_{\rho} \text{  for all } \|u\|_{s,G}=\rho.$$
 \item [(ii)]
By \eqref{f3}, there exist $C_4>0 $ such that
  \begin{equation}\label{pii}
 F(t)\geq C_4 \vert t\vert^\theta, \text{ for all }|t|>1.
 \end{equation}
 Let $\psi  \in C_0^\infty(\R^N)$ with $\|\psi\|_{s,G}>\rho,$ $\|\psi\|_{G_{\frac{N}{s}}}>0$ and $\displaystyle{\int_{\R^N} \vert \psi\vert^\theta dx>0}$. Using Lemma \ref{modulars-norms} and \eqref{pii}, we get
\begin{align*}
I_\l(t\psi)&\leq G_\infty(t)G_\infty(\|\psi\|_{s,G})-
 C_4t^\theta\int_{\R^N} \vert \psi\vert^\theta dx-C_5\min\{t^{p^-_*},t^{p^+_*}\}\min\{\|\psi\|_{G_{\frac{N}{s}}}^{p^-_*},\|\psi\|_{G_{\frac{N}{s}}}^{p^+_*}\}\\
 &\leq t^{p^+}G_\infty(\|\psi\|_{s,G})-
 C_4t^\theta\int_{\R^N} \vert \psi\vert^\theta dx-C_5t^{p^+_*}\min\{\|\psi\|_{G_{\frac{N}{s}}}^{p^-_*},\|\psi\|_{G_{\frac{N}{s}}}^{p^+_*}\},
 \end{align*}
 for some $C_5>0$ and for all $t>1.$
 Since $p^+<\theta$, it follows that,
$$I_\l(t\psi)\rightarrow -\infty\text{ as } t\rightarrow+\infty.$$
Hence, there exists $t_0 > 0$ such that $I_\l(e)<  \delta_{\rho}$ and $\|e\|_{G_{\frac{N}{s}}}>0$, where $e = t_0\psi$. Thus the proof
of (ii) is complete.
\end{itemize}

\end{proof}

For a positive constant $\lambda > 0$, leveraging the geometric insights provided in Lemma \eqref{geo.condition}, we introduce distinctive levels of $I_\lambda$ by
\begin{equation}\label{levels}
  c_\l=\inf_{\gamma\in\Gamma}\max_{t\in[0,1]}I_\l(\gamma(t))
\end{equation}
where $\Gamma=\left\{\gamma\in C([0,1],W^{s,G}(\R^N)):\quad \gamma(0)=0,\ \gamma(1)=e\right\}.$ We present an asymptotic condition pertaining to the level $c_\lambda$. This observation, previously noted in the scalar case (refer to , specifically Lemma 2.2 and Remark 2.3, and \cite{Pucci-G} chapter 1), will play a pivotal role in addressing the challenge posed by the absence of compactness resulting from the presence of critical nonlinearities.

\begin{lemma}\label{cv-levels}
  Let $c_\l$ be the mountain pass energy level of $I_\l$ given by \eqref{levels}. Then
  $$
  \lim_{\lambda\to\infty}c_\lambda=0.
  $$
\end{lemma}

\begin{proof}
  Fix $\l>0$. Let $e$ be the nonnegative function determined in Lemma \ref{geo.condition}. Since the functional $I_\l$ satisfies the mountain pass geometry at $0$ and $e$, there exists $t_\l>0$ verifying $I_\l(t_\l e)=\max_{t\geq1}I_\l(t e).$ Therefore $\langle I_\l^{'}(t_\l e),e\rangle=0.$ Thus
  \begin{align}\label{cv-levels1}
    \langle (-\Delta_g)^s (t_\l e),e \rangle +\int_{\R^N} g(|t_\l e|)e\, dx&= \frac{1}{t_\l}\int_{\R^N} g_*(|t_\l e|)t_\l e\, dx+\lambda\int_{\R^N} f(t_\l e)e\,dx\nonumber\\
    &\geq p^- \min\{t_\l^{p_*^- -1},t_\l^{p_*^+ -1}\}\int_{\R^N} G_{\frac{N}{s}}(| e|)\, dx\nonumber\\
    &\geq C_1 \min\{t_\l^{p_*^- -1},t_\l^{p_*^+ -1}\},
  \end{align}
  by \eqref{f3}, since $\lambda>0.$ From \eqref{G1}, we derive that
  \begin{align}\label{cv-levels2}
    \langle (-\Delta_g)^s (t_\l e),e \rangle +\int_{\R^N} g(|t_\l e|)e\, dx&\leq \frac{p^+}{t_\l}G_\infty(t_\l)G_\infty(\|e\|_{s,G})\nonumber\\
    &\leq C_2 \max\{t_\l^{p^- -1},t_\l^{p^+ -1}\}.
  \end{align}
  Therefore, for $t_\l>1,$ \eqref{cv-levels1} and \eqref{cv-levels2} imply that
  $$
  t_\l^{p_*^- -p^+}\leq C_3\quad\text{for any}\quad\l>0.
  $$
  It follows that $(t_\l)_{\l>0}$ is bounded.

  Fix now a sequence $\left(\lambda_k\right)_k \subset \mathbb{R}^{+}$such that $\lambda_k \rightarrow \infty$ as $k \rightarrow \infty$. Obviously, $\left(t_{\lambda_k}\right)_k$ is bounded. Thus, there exist a $t_0 \geq 0$ and a subsequence of $\left(\lambda_k\right)_k$, still denoted by $\left(\lambda_k\right)_k$, such that $t_{\lambda_k} \rightarrow t_0$ as $k \rightarrow \infty$. Also by \eqref{cv-levels1} and \eqref{cv-levels2}, there exists $C>0$ such that for any $k\in \mathbb{N},$
  \begin{equation}\label{cv-levels3}
    \lambda_k\int_{\R^N} f(t_\l e)e\,dx+p^- \min\{t_\l^{p_*^- -1},t_\l^{p_*^+ -1}\}\int_{\R^N} G_{\frac{N}{s}}(| e|)\, dx\leq C.
  \end{equation}
  We assert that $t_0=0$. Otherwise, \eqref{f1}-\eqref{f2} and the dominated convergence theorem yield, as $k \rightarrow \infty$,
$$
\int_{\mathbb{R}^N} f\left(t_{\lambda_k} e\right) e d x \rightarrow \int_{\mathbb{R}^N} f\left(t_0 e\right) e d x>0.
$$
 Therefore, recalling that $\lambda_k \rightarrow \infty$, we get at once that
$$
\lim _{k \rightarrow \infty}\left(\lambda_k \int_{\mathbb{R}^N} f\left(t_{\lambda_k}\right) e\right) e d x+p^- \min\{t_\l^{p_*^- -1},t_\l^{p_*^+ -1}\}\int_{\R^N} G_{\frac{N}{s}}(| e|)\, dx=\infty,
$$
which contradicts \eqref{cv-levels3}. Thus $t_0=0$ and $t_\lambda \rightarrow 0$ as $\lambda \rightarrow \infty$, since the sequence $\left(\lambda_k\right)_k$ is arbitrary.

Now the path $\gamma(t)=t e, t \in[0,1]$, belongs to $\Gamma$, so that Lemma \ref{geo.condition} gives
$$
0<c_\lambda \leq \max _{t \geq 0} I_\l(\gamma(t)) \leq I_\l\left(t_\lambda e\right) \leq p^+ G_\infty(t_\l)G_\infty(\|e\|_{s,G}) \rightarrow 0
$$
as $\lambda \rightarrow \infty$, since $e$ does not depend on $\lambda$. This completes the proof of the lemma.
\end{proof}

Now we are ready to prove crucial properties of the Palais–Smale sequences of $I_\l$ at the special level $c_\l$, that is for any Palais–Smale sequence at level $c_\lambda$ for  $I_\lambda$ , i.e. a sequence $(u_n)_{n\in\mathbb{N}}\subset W^{s,G}(\R^N)$  satisfying
\begin{equation}\label{cv T}
  I_\lambda(u_n)\rightarrow c_\lambda
\end{equation}
and
\begin{equation}\label{cv T'}
  \sup\{ |\langle I_\lambda^{'}u_n,v\rangle|:\ v\in W^{s,G}(\R^N),\ \|v\|_{s,G}\leq 1\} \rightarrow0
\end{equation}
as $n\rightarrow+\infty$ admits a subsequence which is strongly convergent in $W^{s,G}(\R^N)$.

\begin{lemma}\label{cv of PS}
 There exists $\lambda^*>0$ such that the functional $I_\lambda$ satisfies the Palais–Smale condition at any level $c_\lambda<c_\alpha$ $((PS)_{c_\lambda}$ condition, for short$)$ for all $\lambda\geq\lambda^*$,  where
 $$
 c_\alpha:=\left(\frac{p^-_*}{\theta}-1\right)\frac{p^+}{p^-}\min\left\{(SC_\alpha)^{\frac{1}{\alpha}-1},(SC_\alpha)^{\frac{1}{\alpha}-\frac{p^-}{p^+}}, (SC_\alpha)^{\frac{1}{\alpha}-\frac{p^+}{p^-}}\right\},\quad \alpha\in \{p^-_*,p^+_*\},
 $$
 and $C_\alpha=\left(\frac{p^+}{p^-}\right)^{\frac{1}{\alpha}}.$
\end{lemma}

\begin{proof}
  Let $(u_n)_{n\in\mathbb{N}}$ be a Palais–Smale sequence  at level $c_\l$ for $I_\l$.

{\bf \underline{Claim 4}:} $(u_n)_{n\in\mathbb{N}}$ is bounded in $W^{s,G}(\R^N).$

Indeed, there exists $C_1>0$ such that
\begin{align}\label{bound1}
  C_1(1+\|u_n\|_{s,G})&\geq I_\l(u_n)-\frac{1}{\theta}\langle I_\l^{'}(u_n),u_n\rangle,\quad\forall n\in\mathbb{N}\nonumber\\
  &\geq \left(1-\frac{p^+}{\theta}\right)\rho(u_n)+\left(\frac{p^-_*}{\theta}-1\right)\int_{\R^N}G_{\frac{N}{s}}(|u_n|)dx
  +\lambda\int_{\R^N}\left(f(x,u_n)u_n-F(x,u_n)\right)dx\nonumber\\
  &\geq \left(\frac{p^-_*}{\theta}-1\right)\int_{\R^N}G_{\frac{N}{s}}(|u_n|)dx.
\end{align}
Consequently, for $n$ sufficiently large there exists $C_2>0$ such that
\begin{align*}
  \rho(u_n)&\leq I_\l(u_n)+\int_{\R^N}G_{\frac{N}{s}}(|u_n|)dx+\lambda\int_{\R^N}F(x,u_n)dx\\
  &\leq  I_\l(u_n)+\int_{\R^N}G_{\frac{N}{s}}(|u_n|)dx+\l\e+\int_{\R^N}G(|u_n|)dx+C_\e\int_{\R^N}M(|u_n|)dx\\
  &\leq C_2(1+\|u_n\|_{s,G}).
\end{align*}
If $\|u_n\|_{s,G}>1,$ from Lemma \ref{modulars-norms}, it follows that
$$
\|u_n\|_{s,G}^{p^-}\leq C_3(1+\|u_n\|_{s,G}),
$$
for some $C_3>0.$ Since $p^->1,$ we conclude that $(u_n)_{n\in\mathbb{N}}$ is bounded in $W^{s,G}(\R^N)$.

From the Claim, up to a subsequence, we can assume that $\left\{u_n\right\}_{n \in \mathbb{N}}$ weakly converges to some $u \in W^{s, G}\left(\mathbb{R}^N\right)$.

Let $\mu, \nu, \mu_i, \nu_i, \mu_{\infty}, \nu_{\infty}$ be as in the concentration-compactness principle Theorem \ref{CCP} when applied to $\left\{u_n\right\}_{n \in \mathbb{N}}$.

By Lemma \ref{cv-levels}, there exists $\l_*>0$ such that
\begin{equation}\label{PS4}
  c_\l<c_\alpha\quad\text{for all}\quad \l\geq \l_*.
\end{equation}

We will show that $I=\emptyset$ and $\nu_\infty=0.$

From \eqref{bound1}, we have
$$
c_\l=\limsup_{n\to\infty} I_\l(u_n)-\frac{1}{\theta}\langle I_\l^{'}(u_n),u_n\rangle\geq \limsup_{n\to\infty}\left(\frac{p^-_*}{\theta}-1\right)\int_{\R^N}G_{\frac{N}{s}}(|u_n|)dx.
$$
Combining this with \eqref{CCP4} gives
\begin{equation}\label{PS1}
  c_\l\geq \left(\frac{p^-_*}{\theta}-1\right)(\nu(\R^N)+\nu_\infty).
\end{equation}
Suppose that $I\neq\emptyset,$ there exists at least $i\in I.$

\begin{align*}
\circ_n(1)&=\langle I_\l^{'}(u_n),u_n\phi_{\e,i}\rangle\\
&\leq p^+\int_{\R^N}\phi_{\e,i}(x)\mathcal{D}^s u_n(x)dx-p^-\int_{\R^N}\phi_{\e,i}(x)G_{\frac{N}{s}}(|u_n|)dx\\
&-\lambda\int_{\R^N}f(x,u_n)u_n\phi_{\e,i}dx+\int_{\R^N}\int_{\R^N} g(|D^s u_n|) \frac{D^s u_n}{|D^s u_n|} u_n(y) D^s \phi_{\e,i}\frac{dxdy}{|x-y|^N}.
\end{align*}
This yields
\begin{equation}\label{PS2}
  \left| p^+\int_{\R^N}\phi_{\e,i}(x)\mathcal{D}^s u_n(x)dx-p^-\int_{\R^N}\phi_{\e,i}(x)G_{\frac{N}{s}}(|u_n|)dx\right|\leq \limsup_{n\to\infty}\lambda|I_{1,n,\e}|+\limsup_{n\to\infty}|I_{2,n,\e}|,
\end{equation}
where
$$
I_{1,n,\e}:=\int_{\R^N}f(u_n)u_n\phi_{\e,i}dx,
$$
and
$$
I_{2,n,\e}:=\int_{\R^N}\int_{\R^N} g(|D^s u_n|) \frac{D^s u_n}{|D^s u_n|} u_n(y) D^s \phi_{\e,i}\frac{dxdy}{|x-y|^N}.
$$
From \eqref{f1} and \eqref{G1}, given $\beta> 0$, there is $\delta > 0$ such that
$$ tf(t)\leq \beta p^+ G(\vert t\vert ),\quad \text{for all}\quad\vert t\vert\leq \delta.$$
From \eqref{f2} and \eqref{m1}, there exists $C_\beta>0$ such that
$$ tf(t)\leq C_\beta m^+ M(\vert t\vert),\quad \text{for all}\quad\vert t\vert\geq \delta.$$
Hence,
$$\int_{\R^N}f(u_n)(u_n)\,dx\leq \beta p^+ \int_{\R^N}G(\vert u_n\vert )\,dx+C_\beta m^+\int_{\R^N}M(u_n)\,dx.$$

 Fixing $D_1=\displaystyle{ \sup_{n\in\mathbb{N}} \left( \int_{\R^N}G(\vert u_n\vert )\,dx\right)}$, and $D_2=\displaystyle{ \sup_{n\in\mathbb{N}} \left( \int_{\R^N}M(\vert u_n\vert )\,dx\right)}$, by \eqref{Inequality-embedding} we have
$$\int_{\R^N}f(u_n)(u_n)\,dx\leq\beta  p^+ D_1+C_\beta m^+ D_2=:c.$$
Then,
$$
|I_{1,n,\e}|\leq c\|\phi_{\e,i}\|_{L^{\infty}(B_{2\e}(x_i))}.
$$

From the previous inequality we obtain
\begin{equation}\label{I1n0}
  \limsup_{\e\to 0^+}\limsup_{n\to\infty} |I_{1,n,\e}|=0.
\end{equation}
Let $\beta>0$ be arbitrary and fixed. By \eqref{G1}, \eqref{ineb}, Young inequality, and  Claim 2, we obtain
\begin{align*}
  |I_{2,n,\e}| & \leq p^+\int_{\R^N}\int_{\R^N} g(|D^s u_n|) |u_n(y)| |D^s \phi_{\e,i}|\frac{dxdy}{|x-y|^N}\\
  &\leq \beta p^+\int_{\R^N}\int_{\R^N}\tilde{G}\left(g(|D^s u_n|)\right)\frac{dxdy}{|x-y|^N}+
  p^+\int_{\R^N}\int_{\R^N} G\left(|u_n(y)| |D^s \phi_{\e,i}|\right)\frac{dxdy}{|x-y|^N}\\
  &\leq \beta p^+\int_{\R^N}\int_{\R^N}G(|D^s u_n|))\frac{dxdy}{|x-y|^N}+
  p^+\int_{\R^N}\int_{\R^N}G_\infty(|u_n(y)|) G\left(|D^s \phi_{\e,i}|\right)\frac{dxdy}{|x-y|^N}\\
  &\leq c\beta+p^+\int_{\R^N}\int_{\R^N}G_\infty(|u_n(y)|) G\left(|D^s \phi_{\e,i}|\right)\frac{dxdy}{|x-y|^N}
\end{align*}
By considering the limit superior in the previous inequality as $n\to\infty$ , followed by taking the limit superior as $\e\to0^+$ and incorporating Claim 2, we reach the following conclusion:
\begin{equation}\label{I2n0}
  \limsup_{\e\to 0^+}\limsup_{n\to\infty} |I_{2,n,\e}|=0.
\end{equation}

By considering the limit superior in \eqref{PS2} as $\e\to0^+$ and incorporating the conditions specified in equations \eqref{I1n0} and \eqref{I2n0}, we can derive
$$
p^+\mu_i=p^-\nu_i.
$$
Plugging this into \eqref{CCP3} we obtain
\begin{align*}
&S\min\left\{\left(\frac{p^+}{p^-}\right)^{\frac{1}{p^-_*}}\mu_i^{\frac{1}{p^-_*}},
\left(\frac{p^+}{p^-}\right)^{\frac{1}{p^+_*}}\mu_i^{\frac{1}{p^+_*}}\right\}\\
&\leq\max \left\{ \mu_i, \ \mu_i^{\frac{p^-}{p^+}}, \ \mu_i^{\frac{p^+}{p^-}} \right\}.
\end{align*}
Then,
\begin{equation}\label{PS3}
  \nu_i=\frac{p^+}{p^-}\mu_i\geq \frac{p^+}{p^-}\min\left\{(SC_\alpha)^{\frac{1}{\alpha}-1},(SC_\alpha)^{\frac{1}{\alpha}-\frac{p^-}{p^+}}, (SC_\alpha)^{\frac{1}{\alpha}-\frac{p^+}{p^-}}\right\},
\end{equation}
where $\alpha\in \{p^-_*,p^+_*\}$ and $C_\alpha=\left(\frac{p^+}{p^-}\right)^{\frac{1}{\alpha}}.$

From \eqref{PS1} and \eqref{PS3}, we derive $c_\lambda\geq c_\alpha$ for all $\lambda>\lambda_*$ which is in contradiction with \eqref{PS4}, and hence $I=\emptyset.$

Argue in a similar fashion to establish $I=\emptyset$. By leveraging $\phi_R$, we can see that $\nu_\infty=0.$

Combining the facts that$I=\emptyset$ and $\nu_\infty=0$ with  \eqref{CCP1} and \eqref{CCP4}, we obtain
$$
\limsup_{n\to\infty}\int_{\R^N}^{} G_{\frac{N}{s}} (|u_n|)  dx=\int_{\R^N}^{} G_{\frac{N}{s}} (|u|)  dx.
$$
Invoking Fatou's Lemma, we get
$$
\int_{\R^N}^{} G_{\frac{N}{s}} (|u|)  dx\leq \liminf_{n\to\infty}\int_{\R^N}^{} G_{\frac{N}{s}} (|u_n|)  dx.
$$
Thus,
$$
\lim_{n\to\infty}\int_{\R^N}^{} G_{\frac{N}{s}} (|u_n|)  dx=\int_{\R^N}^{} G_{\frac{N}{s}} (|u|)  dx.
$$
By \cite[Lemma 3.4]{Bonder3}, we deduce that
$$
\lim_{n\to\infty}\int_{\R^N}^{} G_{\frac{N}{s}} (|u_n-u|)  dx=0,
$$
and so,
\begin{equation}\label{PS6}
  u_n\to u\quad \text{ in}\quad L^{G_\frac{N}{s}}(\R^N).
\end{equation}
Consequently, we have
\begin{equation}\label{PS5}
  \lim_{n\to\infty}\int_{\R^N}g_*(|u_n|)\frac{u_n}{|u_n|}(u_n-u)dx=0.
\end{equation}
Indeed, by H\"{o}lder inequality, \eqref{ineb} and the boundness of $(u_n)$ in $W^{s,G}(\R^N),$ we derive
\begin{align*}
   \left|\int_{\R^N}g_*(|u_n|)\frac{u_n}{|u_n|}(u_n-u)dx\right|&\leq 2\|g_*(|u_n|)\|_{\widetilde{G}_\frac{N}{s}}\|u_n-u\|_{G_\frac{N}{s}}\\ &\leq \left(\int_{\R^N}\widetilde{G}_\frac{N}{s}(g_*(|u_n|))dx\right)^{\frac{1}{\iota}}\|u_n-u\|_{G_\frac{N}{s}}\\
   &\leq  \left(\int_{\R^N}G(|u_n|))dx\right)^{\frac{1}{\iota}}\|u_n-u\|_{G_\frac{N}{s}}\\
   &\leq C\|u_n-u\|_{G_\frac{N}{s}},
\end{align*}
for some $C>0,$ here the constant $\iota>1$ can be deduced from \cite[Lemma 4.5]{Hlel1}. We deduce \eqref{PS5} from \eqref{PS6}.

On one hand, by \eqref{cv T'}, we find that
\begin{align*}
  \lim_{n\to\infty} \langle I_\l^{'}u_n,u_n-u\rangle=0.
\end{align*}
On the other hand, by \eqref{f1}- \eqref{f2}, Young inequality, boundedness of $(u_n)_n$ and \cite[Lemma 2.9]{Bonder-Salort1}, we deduce that
\begin{align*}
   \int_{\R^N}f(u_n)(u_n-u)\,dx\to0,\quad\text{as}\ n\to\infty.
\end{align*}
Since
\begin{align*}
\langle I_\l^{'}u_n,u_n-u\rangle&=\langle\Phi_{s,G}^{'}(u_n),u_n-u\rangle
+\langle\Phi_G^{'}(u_n),u_n-u\rangle\\
&-\int_{\R^N}g_*(|u_n|)\frac{u_n}{|u_n|}(u_n-u)dx-\lambda\int_{\R^N}f(u_n)(u_n-u)\,dx,
\end{align*}
then
$$
\langle\left(\Phi_{s,G}+\Phi_G\right)^{'}(u_n),u_n-u\rangle
\to0,\quad\text{as}\quad n\to\infty.
$$
According to \cite[Lemma 21]{Sabri2}, $(\Phi_{s,G}+\Phi_G)$ is of type $(S_+)$ and so $(u_n)$ converges strongly to $u$ in $W^{s,G}(\R^N).$

\end{proof}

\begin{proof}{{\bf Proof of Theorem \ref{Existence}}}
 The proof is essentially a synthesis of Lemmas \ref{geo.condition} and \ref{cv of PS}, and Mountain Pass Theorem.
\end{proof}

\end{document}